\documentclass[a4paper,11pt,reqno]{amsart}
\usepackage[linktocpage]{hyperref}
\usepackage{amssymb, paralist, xspace, graphicx, url, amscd, euscript, verbatim, mathrsfs,stmaryrd,epic,eepic,color,mathtools,tikz-cd,amsthm}
\usepackage[all]{xy}
\usepackage{indentfirst}
\numberwithin{equation}{section}
\newtheorem{theorem}[subsection]{Theorem}
\newtheorem{proposition}[subsection]{Proposition}
\newtheorem{conjecture}[subsection]{Conjecture}
\newtheorem{corollary}[subsection]{Corollary}
\newtheorem{lemma}[subsection]{Lemma}

\theoremstyle{definition}
\newtheorem{remark}[subsection]{Remark}

\newtheorem{definition}[subsection]{Definition}

\newtheorem*{claim}{Claim}
\newtheorem*{claimproof}{Proof of Claim}

\newcommand{\rS}{{\rm S}}
\newcommand{\rH}{{\rm H}}
\newcommand{\tY}{\tilde{Y}}
\newcommand{\tI}{\tilde{I}}
\newcommand{\cO}{\mathcal{O}}
\newcommand{\cA}{\mathcal{A}}
\newcommand{\cE}{\mathcal{E}}
\newcommand{\cF}{\mathcal{F}}
\newcommand{\cU}{\mathcal{U}}
\newcommand{\PP}{\mathbb{P}}
\newcommand{\QQ}{\mathbb{Q}}
\newcommand{\ZZ}{\mathbb{Z}}
\newcommand{\LL}{\mathbb{L}}
\newcommand{\GG}{\mathbb{G}}
\newcommand\CH{\operatorname{CH}}
\newcommand\Ext{\operatorname{Ext}}
\newcommand\Gr{\operatorname{Gr}}
\newcommand\pr{\operatorname{pr}}
\newcommand\ch{\operatorname{ch}}
\begin{document}

\title[One-cycles on Gushel-Mukai]{One-cycles on Gushel-Mukai fourfolds and the Beauville-Voisin filtration}
\date{\today}

\author{Ruxuan Zhang}
\address{Ruxuan Zhang, Shanghai Center for Mathematical Sciences, Fudan University, Jiangwan Campus, Shanghai, 200438, China}
\email{rxzhang18@fudan.edu.cn}

\begin{abstract}
We prove that the invariant locus of the involution associated with a general double EPW sextic is a constant cycle surface and introduce a filtration on $\CH_1$ of a Gushel-Mukai fourfold. We verify the sheaf/cycle correspondence for sheaves supported on low-degree rational curves, parallel to the cubic fourfolds case of Shen-Yin's work.
\end{abstract}


\maketitle

\section{Introduction}
A Gushel-Mukai (GM) fourfold is a smooth prime Fano fourfold $X\subset \PP^8$ of degree $10$ and index $2$. Throughout the paper, a GM fourfold is always an ordinary GM fourfold, 
which can be obtained as a smooth dimensionally transverse intersection 
$$ X=\Gr(2,5)\cap \PP^8\cap Q\subset \PP^9,$$ where $\Gr(2,5)\subset \PP^9$ is the Pl\"ucker embedding, $\PP^8$ is a linear subspace and $Q$ is a quadric hypersurface. 

GM fourfolds share many properties with cubic fourfolds. One of the distinguished facts is that there is a parallel between GM and cubic fourfolds from a categorical viewpoint. 
Similarly to the semiorthogonal decomposition of the derived category of a cubic fourfold, Kuznetsov and Perry showed in \cite{kuznetsov2018derived} that the derived category of a GM fourfold admits a semiorthogonal decomposition:
$${\rm D}^b(X)= \langle \cA_X,\cO_X,\cU_X^{\vee},\cO_X(1),\cU_X^{\vee}(1)\rangle,$$ where $\cU_X$ is the tautological bundle of $\Gr(2,5)$ restricted to $X$ and $\cA_X$ is a K3 category.

Another fact is that a GM fourfold has an associated irreducible holomorphic symplectic (IHS) variety, which is called the dual double EPW sextic, compared with the Fano variety of lines associated with a cubic fourfold, see \cite{debarre2018gushel}\cite{o2006irreducible}\cite{iliev2011fano} for details.

The Fano variety of lines of a cubic fourfold and the dual double EPW sextic of a GM fourfold both admit constructions via moduli spaces of stable objects and Hilbert schemes of rational curves.
In \cite[Proposition 5.17]{perry2019stability}, it has been shown that at least one of the double EPW sextic and the dual double EPW sextic can be realized as a moduli space of stable objects of $\cA_X$ and conjecturally both are. 
In \cite{iliev2011fano}, Iliev and Manivel gave a more geometrical description of the dual double EPW sextic.  For a general GM fourfold $X$,
let $F(X)$ be the Hilbert scheme of conics on $X$ and $\tY^\vee$ be the associated dual double EPW sextic of $X$, there exists a morphism $$F(X)\rightarrow \tY^\vee$$ such that a general fiber is $\PP^1$.

Quite generally, for a $2n$-dimensional IHS variety $M$, Voisin predicted in \cite{voisin2016remarks} that there exists a filtration of the Chow ring $\CH^*(M)$ (called Beauville–Voisin filtration), which can be viewed as an opposite to the conjectural Bloch-Beilinson filtration.
For 0-cycles, the filtration $\rS_\bullet\CH_0(M)$ is defined by
$$\rS_{i}\CH_0(M):=\left<x\in M|~\dim O_y\geq n-i\right>,$$ where $O_y$ is the orbit of $y$ under the rational equivalence (here we follow the definition in \cite{shen2020categories}, which is opposite to Voisin's original definition). Let $Z\subset M$ be a subvariety. If any two points on $Z$ have the same class in $\CH_0(M)$, we call that $Z$ is a constant cycle variety on $M$.
In particular, for an IHS variety of dimension $4$, the filtration is determined by constant cycle surfaces and uniruled divisors.
Let $\rS_iM\subset M$ be the set of points with dim $O_y\geq n-i$. Voisin conjectured in \cite[Conjecture 0.4]{voisin2016remarks}  that
\begin{equation}\label{vconj}
    \mathrm{dim}~\rS_i M= n+i.
\end{equation}
It is closely related to the existence of algebraically coisotropic subvarieties on $M$.

Let $M$ be a nonsingular projective moduli space of stable objects on a K3 category $\cA$, then it is an IHS variety in many cases.
Shen and Yin predicted that all the Beauville–Voisin filtrations associated with different moduli spaces of stable objects on a fixed K3 category should be controlled by a universal filtration on the Grothendieck group of the K3 category, see  \cite[Speculation 0.1]{shen2020categories}. 

When the K3 category is the derived category of a K3 surface $S$ and $M$ is a $2d$-dimensional nonsingular moduli space of stable objects, the O'Grady's filtration on $\CH_0(S)$
$$\rS_i(S) :=\bigcup_{\substack{\deg([z])=i\\z ~\text{effective}}}\{ [z] + \ZZ \cdot [o_S] \}$$
serves as the universal filtration for all moduli space of stable objects. It has been shown in \cite{shen2020derived} that 
$$c_2(\cE)\in \rS_d(S)$$ for any $\cE\in M$ and that $c_2(\cE)\in \rS_i(S)$ implies  $\cE\in \rS_i\CH_0(M)$.

Another example arises when the K3 category is the Kuznetsov component $\cA_X$ of a cubic fourfold $X$. Using the Beauville-Voisin filtration on the Fano variety of lines on $X$, Shen-Yin introduced in \cite{shen2020derived} a filtration $\rS_{\bullet}(X)$ on $\CH_1(X)$ induced by the incidence correspondence. They conjectured that 
the filtration $\rS_{\bullet}(X)$ may serve as a universal filtration on $K_0(\cA_X)$, i.e.,
for $\cE\in \cA_X$, there should be $$c_3(\cE)\in \rS_d(X),$$  where $d=\frac{1}{2}\dim \Ext^1(\cE,\cE)$. This implies  \eqref{vconj}  by a standard argument and they verified the prediction when $\cE$ is supported on low-degree rational curves.

It is natural to ask whether there exists a universal filtration on $\cA_X$ for a GM fourfold, similarly to the case of a cubic fourfold. We first study the Beauville–Voisin filtration on $\CH_0(\tY^\vee)$ by finding a natural constant cycle surface on $\tY^\vee$. The invariant locus $Z$ of the involution associated with the dual double EPW sextic is a surface of general type, whose $\CH_0$ might be huge. However, it was asked whether $Z$ is a constant cycle surface on $\tY^\vee$, see \cite[Remark 2.7]{laterveer2021zero}. 
Our first result is (Theorem \ref{constant}):

\begin{theorem}\label{mainthm1}
The invariant locus $Z$ of the involution associated with a general dual double EPW sextic $\tY^\vee$ is a constant cycle surface. If $\tY^\vee$ is very general, any other constant cycle surface $Z'$ meets $Z$.
\end{theorem}
The $0$-piece of the Beauville–Voisin filtration of $\tY^\vee$ is represented by a point on $Z$ and the $1$-piece is represented by a point on a uniruled divisor. Theorem \ref{mainthm1} allows us to construct a filtration  $\rS_{\bullet}(X)$ on $\CH_1(X)$ by the incidence correspondence of Iliev and Manivel's construction. For a very general GM fourfold, we predict that the filtration should serve as a universal  filtration on $\cA_X$ in the following way:

The Grothendieck group of $\cA_X$ is shown to be generated by $\pr [\cO_c(1)]$ and $\pr [\cO_p]$ for a very general $X$, where $c$ is a conic, $p$ is a point on $X$ and $\pr$ is the projection $$\pr:K_0(X)\rightarrow K_0(\cA_X).$$
Unlike the cubic fourfold case, the dual double EPW sextic determines the GM fourfold only up to period partners and the Chern class $c_3(\pr [\cO_p])$ of a skyscraper sheaf might lie in different pieces of the filtration with respect to different GM fourfolds. We define a modification $$p:K_0(\cA_X)\rightarrow \CH_1(X)$$ of $c_3$ by dropping the effect of $\pr [\cO_p]$, see Section \ref{S5} for details. 

We propose the following conjecture relating the $K3$ category $\cA_X$ to the filtration $\rS_\bullet(X)$ as a parallel to the case of a cubic fourfold.

\begin{conjecture} \label{mainconj}
For any object $\cE \in \cA_X$, we have
\[
p(\cE) \in \rS_{d(\cE)}(X).
\]
\end{conjecture}
Let $i^*$ be the left adjoint functor of the natural inclusion $\cA_X\hookrightarrow D^b(X)$. We show that the classes of lines and conics on a very general GM fourfold $X$ are in $\rS_2(X)$, which implies 
our second main result concerning sheaves supported on low-degree rational curves (Theorem \ref{maint}):
\begin{theorem}\label{mainthm}
We assume that $X$ is very general.
If $\cF$ is supported on lines or conics on $X$, conjecture \ref{mainconj} holds for $\cE=i^*\cF$. 
\end{theorem}

Our approach to deducing Theorem \ref{mainthm1} has another application, concerning with the relation between $\CH_0(\tY^\vee)$ and $\CH_1(X)$. When $\tY^\vee$ is birational to $S^{[2]}$ for some K3 surface $S$,  it has been shown in \cite{laterveer2021zero} that $\CH_0(\tY^\vee)$ with $\QQ$-coefficients has a natural decomposition for the Chow group. For 0-cycles, the decomposition is
$$\CH_0(\tY^\vee)=\QQ\cdot o\oplus\CH_0(S)_{\rm hom}\oplus\CH_0(\tY^\vee)^+_{\rm hom},$$
where $\CH_0(\tY^\vee)^+_{\rm hom}$ is the $\iota$-invariant part of $\CH_0(\tY^\vee)_{\rm hom}$. We show that in the general case, the group $\CH_1(X)$ would play the role of the group $\CH_0(S)$.
\begin{theorem}(Theorem \ref{chowiso})
When the dual double EPW sextic $\tY^\vee$ is general, we have the following isomorphism of groups:
\begin{equation*}
    \ZZ\cdot o\oplus \CH_0(\tY^\vee)^{-}_{\rm hom} \cong \CH_1(X),
\end{equation*}
where $\CH_0(\tY^\vee)^{-}_{\rm hom}$ is the $\iota$-anti-invariant part of $\CH_0(\tY^\vee)_{\rm hom}$.
\end{theorem}
In particular, the group $\CH_1(X)$ is torsion-free and we see that if $X'$ is another GM fourfold with the same dual double EPW sextic as $X$, then their groups of 1-cycles are isomorphic.

Our results rely on Iliev and Manivel’s construction and hence all dual double EPW sextics are assumed to be general in the rest of the paper.

\section{Notation}
The following notations appear frequently in the paper, so we gather them together for the sake of convenience.
\begin{itemize}
    \item $V_m$ is an $m$-dimensional vector space, $X$ is a smooth ordinary GM fourfold and $\tY^\vee$ is the dual double sextic associated with $X$.
    \item $\iota$ is the involution of $\tY^\vee$ and $Z\subset \tY^\vee$ is the invariant locus of $\iota$. 
    \item $c$ is a conic on $X$, $F(X)$ is the Hilbert scheme of conics on $X$ and $\alpha: F(X)\rightarrow \tY^\vee$ is the birational $\PP^1$-fibration. $\alpha$ induces an isomorphism  $\alpha_* : \CH_0(F(X))\cong \CH_0(\tY^\vee).$
    \item $P$ is the universal conic, $p,q$ are the projections to $F(X)$ and $X$. $\Phi=p_*q^*: \CH^*(X)\rightarrow \CH^*(F)$. $\Psi$ is the composition of $q_*p^*$ and $\alpha_*^{-1}$: $$\Psi=q_*p^*\alpha_*^{-1}: \CH_0(\tY^\vee)\rightarrow \CH_1(X),$$ sending $\alpha(c)$ to its class $[c]\in \CH_1(X)$.
    \item $S_{V_4}=\Gr(2,V_4)\cap X$ for some $V_4\subset V_5$, which is a degree $4$ Del Pezzo surface.
    \item $D_c$ is the subvariety of $F(X)$ parameterizing conics intersecting the given conic $c$. $I\subset F(X)\times F(X)$ is the subvariety of $ F(X)\times F(X)$, parameterizing two intersecting conics.
\end{itemize}
   
\section{Preliminaries}
Let $X=\Gr(2, V_5)\cap H\cap Q$ be a very general GM fourfold and $\tY^\vee\rightarrow Y^\vee$ be its associated double cover of the dual EPW sextic. Denote by $\iota:\tY^\vee\rightarrow \tY^\vee$ the associated involution and by $Z$ the invariant locus of $\iota$. We recall some facts about GM varieties and EPW sextics in this section. Throughout the paper, $V_m$ is an $m$-dimensional vector space. 

\subsection{EPW sextics}
Debarre and Kuznetsov gave a way to associate a GM variety with an EPW sextic in \cite{debarre2018gushel} via linear algebra. We will follow \cite{debarre2020survey} to recall some basic facts about EPW sextics. A GM variety $X_n=\Gr(2,V_5)\cap \PP(W)\cap Q$ of dimension $n$ is a smooth intersection of a linear subspace of dimension $n+4$, a quadratic hypersurface and $\Gr(2,V_5)$ inside $\PP(\bigwedge^2V_5)$.
It has an associated GM data $(W_{n+5}, V_6, V_5, q)$, where 
\begin{itemize}
    \item $V_6$ is a $6$-dimensional vector space, consisting of quadrics containing $X$.
    \item $V_5$ is a hyperplane section of $V_6$, consisting of quadrics containing $\Gr(2,V_5)$.
    \item $W_{n+5}\subset \bigwedge^2V_5$ is a subspace of dimension $n+5$, corresponding to the linear subspace in the definition of GM variety.
    \item $q:V_6\rightarrow \mathrm{Sym}^2W_{n+5}^{\vee}$, is a linear map such that (here we choose an isomorphism $\bigwedge^5V_5\cong \mathbb{C}$)
    $$\forall v\in V_5, q(v)(w,w)=v\wedge w\wedge w.$$
    
\end{itemize}
Then $X=\bigcap_{v\in V_6}Q(v)\subset \PP(W_{n+5})$, where $Q(v)$ is defined by $q(v)$. For a general GM data, $X$ is a smooth GM variety.

On the other hand, $\bigwedge^3V_6$ can be endowed with a symplectic form by wedge product.
Debarre and Kuznetsov defined the Lagrangian data set $(V_6,V_5,A)$ with $A\subset \bigwedge^3V_6$ a Lagrangian subspace and $V_5\subset V_6$ a hyperplane. It gives an EPW data $$Y_A^{\geq l}:=\{[v]\in \PP(V_6)~|~\mathrm{dim}(A\cap (v\wedge\bigwedge ^2V_6))\geq l\}$$ and an extra $V_5$. Here, $Y_A^{\geq 1}$ is a sextic in $\PP(V_6)$ and $Y_A^{\geq 2}$ is a surface. 
In \cite{o2013double}, O'Grady showed that

\begin{theorem}
If $A$ contains no decomposable vectors and $Y_A^{\geq 3}=\emptyset$, there is a double cover $\tY_{A}\rightarrow Y_A^{\geq 1}$ branched along the surface $Y_A^{\geq 2}$. $\tY_A$ is a smooth IHS variety, called the double EPW sextic.
\end{theorem}
Iliev and Manivel described in \cite{iliev2011fano} another way to construct the EPW sextic associated with a GM variety $X$. Define $\widetilde{\mathrm{Disc}}(X)$ to be the subscheme of $\PP(V_6)$ of singular quadrics containing $X$. It is a hypersurface of degree $n+5$ and the multiplicity of the hyperplane $\PP(V_5)$ of Pl\"ucker quadrics is at least $n-1$. Then $$\mathrm{Disc}(X):= \widetilde{\mathrm{Disc}}(X)-(n-1)\PP(V_5)$$ equals the EPW sextic $Y_A$.

One can use the same construction for the  Lagrangian subspace  $A^{\vee}\subset \bigwedge^3V_6^\vee$, the corresponding $\tY_{A^{\vee}}$ is called the dual double EPW sextic. We mainly deal with the dual double EPW sextic associated with a GM variety and abbreviate it as $\tY^\vee$ in the following sections.

There is a bijection between the GM data sets and the Lagrangian data sets by  \cite[Theorem 2.2]{debarre2020survey}. 
However, the GM  fourfolds associated with a fixed dual EPW sextic are not unique. Debarre and Kuznetsov proved the following \cite{debarre2020gushel}\cite[Theorem 2.6]{debarre2020survey}:
\begin{theorem}\label{moduli}
Let $M^{GM}$ and $M^{EPW}$ be the coarse moduli spaces of ordinary GM $n$-folds and EPW sextics. Then there exists a surjective morphism:
\begin{equation*}
    \pi : M^{GM} \longrightarrow M^{EPW}
\end{equation*}
with fibre $\pi^{-1}(A^{\vee})=Y_{A}^{5-n}$. 
\end{theorem}
If two GM fourfolds are in the same fiber of $\pi$, we say that they are {\it period partners}.

\subsection{Lines and Conics on \texorpdfstring{$X$}{X}}
In this subsection, we gather some facts about lines and conics on a general GM fourfold $X$ following \cite[Section 3]{iliev2011fano} and \cite{debarre2012period}.

Let $F(X)$ be the Hilbert scheme of conics and $F_1(X)$ the Hilbert scheme of lines on $X$. $F(X)$ is a smooth fivefold.
For a conic $c\subset X$, we denote by $\langle c\rangle$ the plane spanned by $c$.  For $V_4\subset V_5$, we set $$S_{V_4}= \Gr(2,V_4)\cap H\cap Q \subset \PP(\bigwedge^2V_4)\cap H\cong \PP^4,$$ which is a degree $4$ Del Pezzo surface.

The hyperplane $H$ can be viewed as a two-form $\omega$ on $V_5$ with a one-dimensional kernel $W_1$. Let $V_i\subset V_5$ be a subspace of dimension $i$. There are three types of conics on $X$:
\begin{enumerate}
    \item a {\bf $\tau$-conic}: The plane $\langle c\rangle$ spanned by the conic $c$ is not contained in $\Gr(2,V_5)$. The normal bundle of a smooth  $\tau$-conic $c$ is $N_{c/X}=\cO_c\oplus\cO_c(1)\oplus\cO_c(1)$. A general conic is a $\tau$-conic.
    \item a {\bf $\rho$-conic}: The plane $\langle c\rangle$ is of the form $\PP(V_3^{\vee})$ and contained in $H$, or equivalently, $W_1\subset V_3$ and $V_3$ is isotropic for $\omega$. 
    For a general $X$, the normal bundle of a smooth  $\rho$-conic $c$ is $N_{c/X}=\cO_c\oplus\cO_c(1)\oplus\cO_c(1)$. They form a $3$-dimensional subvariety in $F(X)$. 
    \item a {\bf $\sigma$-conic}: The plane $\langle c\rangle$ is of the form $\PP(V_1\bigwedge V_4)$ and contained in $H$, or equivalently, $V_4\subset V_1^{\perp_\omega}$. For $V_1\neq W_1$, dim $V_1^{\perp_\omega}=4$. They are parameterized by $Bl_{[W_1]}(\PP(V_5))$ and form a $4$-dimensional subvariety in $F(X)$.
\end{enumerate}
We see that all the $\rho$-conics (resp. $\sigma$-conics) are represented by the same class in $\CH_1(X)$ and denote it by $\rho$ (resp. $\sigma$) for abuse of notations.

Lines on $X$ are related to the double EPW sextic $\tY_A$.
The general $\sigma$-conics which are unions of two lines are parameterized by $$[V_1]\in Y_{A}^{\geq 1} \cap \PP(V_5)$$ and a general line is one of the components of a degenerate $\sigma$-conic by \cite[Theorem 4.7]{debarre2012period}. More precisely, the Hilbert scheme of lines $F_1(X)$ is the blow-up of $\tY_{A}\times _{\PP(V_6)}\PP(V_5)$ at a point.
There is a rational involution  $$\iota_1: F_1(X)\dashrightarrow F_1(X),$$  with $l\cup \iota_1(l)$ a $\sigma$-conic.

The Fano variety of conics $F(X)$ admits birationally a $\PP^1$-fibration to $\tY^\vee$ in the following way. A conic $c\subset X$  spans a degree $2$ surface in $\PP(V_5)$, which is degenerate and hence contained in some $\PP(V_4)$. Therefore, $c\subset  S_{V_4}$. One may check that if $c$ is not of $\rho$-type, the corresponding $V_4$ is unique. 

The linear system $|c|$ in $S_{V_4}$ is one-dimensional, which defines the birational $\PP^1$-fibration to the dual double EPW sextic, see \cite[Proposition 4.18]{iliev2011fano}: 
 $$\alpha : F(X)\rightarrow \tY^\vee.$$ The fibration maps all $\sigma$ and $\rho$ conics to two points $y_1$ and $y_2=\iota(y_1)$ on $\tY^\vee$ and further to the Pl{\"u}cker hyperplane $[V_5]\in Y\subset \PP(V_6^{\vee})$, by \cite[Lemma 4.12]{iliev2011fano}.
In conclusion, we have the following diagram.
\begin{equation}\label{notation}
    \xymatrix{
  & P \ar[d]^q \ar[r]^p & F(X) \ar[r]^\alpha &\tY^\vee            \\
  & X                       },
\end{equation}
where $P$ is the universal conic. The fiber of $\alpha$ for $y\neq y_1,y_2$ is $\PP^1$ and hence $\alpha$ induces an isomorphism  $$\alpha_* : \CH_0(F(X))\cong \CH_0(\tY^\vee).$$ 
We will frequently use some maps between Chow groups:  Let $\Phi=p_*q^*: \CH^*(X)\rightarrow \CH^*(F)$ and $\Psi$ be the composition of $q_*p^*$ and $\alpha_*^{-1}$: $$\Psi=q_*p^*\alpha_*^{-1}: \CH_0(\tY^\vee)\rightarrow \CH_1(X),$$ sending $\alpha(c)$ to its class $[c]\in \CH_1(X)$.

The involution $\iota$ can be described generically as follows.  For $y\in \tY^\vee$, we take a conic $c$ in the fiber $\alpha^{-1}(y)$ and it is contained in $S_{V_4}\subset \PP^4$. Then taking a hyperplane in $\PP^4$ containing $c$, it intersects with $S_{V_4}$ along another conic $c'$. Then $$\iota(y)=\alpha(c').$$ One can check that for a different choice of $c$ and a hyperplane in $\PP^4$, the corresponding $\alpha(c')$ is the same.  

Finally, we need to describe the Hilbert scheme of conics $F(X_3)$ on a linear section  $X_3$ of $X$.  Let  $B\subset \bigwedge^3V_6$ be the Lagrangian subspace corresponding to $X_3$. Then $$Y^{\geq 2}_{B^{\vee}}\subset Y^{\geq 1}_{A^{\vee}}$$ and $F(X_3)$ is birational to the pull-back of $Y^{\geq 2}_{B^{\vee}}$ in $\tY^\vee$, see \cite[Section 5.1]{iliev2011fano}.

\begin{remark}
In \cite[Section 2.4]{voisin2018triangle}, it was claimed that the locus of  $\sigma$-conics is contracted to a constant cycle surface on $\tY^\vee$. However, it is not true, because, by Iliev and Manivel's construction, the locus of  $\sigma$-conics should be contracted to a point. 
\end{remark}

\subsection{The Beauville-Voisin conjecture}
The Beauville-Voisin conjecture implies that the cycle class map of an IHS variety is injective when restricted to the subring generated by divisors in the Chow ring. 

For a double EPW sextic, the Beauville-Voisin conjecture for zero cycles was proved in \cite{ferretti2012chow} and \cite{laterveer2021zero}. The result is slightly stronger, that is,  the cycle class map restricted to the subring \begin{equation}\label{weak splitting}
    R^*(\tY^\vee):= \langle \CH^1(\tY^\vee), \CH^2(\tY^\vee)^+, c_j(\tY^\vee)\rangle 
\end{equation}
generated by divisors, $\iota$-invariant $2$-cycles and Chern classes is injective for $i=4$. 
In particular, there is a distinguished 0-cycle $o$ on $\tY^\vee$. For a very general double EPW sextic,  $\CH^1(\tY^\vee)$ is generated by $h$, which is the polarization induced by the double cover of the sextic hypersurface $Y\subset \PP^5$.  Let $Z$ be the invariant locus of the involution.
The following relations have been obtained in \cite[Proposition 6.1]{ferretti2012chow}\cite[Lemma 9.3]{ferretti2012chow}:  
\begin{equation}\label{intersection number}
    3Z=15h^2-c_2, Z^2=192, h^4=12, h^2Z=40.
\end{equation}
Moreover, the classes of $Z^2,h^4$ and $h^2Z$ are all proportional to the class $o$.

\section{A constant cycle surface on \texorpdfstring{$\tY^\vee$}{Y}}
This section is devoted to proving Theorem \ref{mainthm1}. In the following parts of the paper, we assume that $X$ is general.

\subsection{Two facts about \texorpdfstring{$\CH_1(X)$}{X}}
We first show that $\CH_1(X)$ is rationally generated by conics on $X$ and those conics represented by points on the invariant surface $Z$ have the same class in $\CH_1(X)$.  
\begin{proposition}\label{ch}
The group $\CH_1(X)_{\QQ}$ is generated by conics and   $\CH_1(X)$  is generated by lines on $X$.
\end{proposition}

\begin{proof}
We first show that $X$ is connected by chains of conics. Every two general points $x_1,x_2$ on $X$ correspond to two $2$-dimensional subspaces in $V_5$, hence the two subspaces are in some $4$-dimensional $V_4\subset V_5$. Thus the two points are contained in $S_{V_4}$, which is a degree $4$ Del Pezzo surface. Choosing two conic fibrations of $S_{V_4}$, we see that the $x_1$ and $x_2$ are connected by a chain of $3$ conics. Hence by \cite[Proposition 3.1]{tian2014one}, there is a non-zero integer $N$, such that for any 1-cycle $D$ on $X$, $$N\cdot D\sim {\rm sum~of~conics}.$$ Hence $\CH_1(X)_{\QQ}$ is generated by conics.

For integral coefficients, it has been shown in \cite[Theorem 1.3]{tian2014one} that $\CH_1(X)$ is generated by rational curves. By bend and break, a rational curve $C$ of degree $\geq 3$ on $X$ is algebraically equivalent to a sum of lower-degree rational curves. Moreover, a conic on a Del Pezzo surface is linearly equivalent to a sum of lines, therefore $C$ is algebraically equivalent to a sum of lines $\sum_i n_iL_i$. The argument in \cite[Lemma 3.4]{tian2014one} shows that $\CH_1(X)_{\rm alg}$ is divisible. Therefore there exists a 1-cycle $C'$ such that $$N\cdot C'= C-\sum_i n_iL_i \in \CH_1(X)_{\rm alg},$$ since $N\cdot C'$ is rationally equivalent to a linear combination of lines, we find that $\CH_1(X)$ is generated by lines with $\ZZ$-coefficients. 
\end{proof}

\begin{lemma}\label{2.1}
For $y\in \tY^\vee$, we have $\Psi(y+\iota(y))$ is constant in $\CH_1(X)$. In particular, $\Psi(2y)$ is constant in $\CH_1(X)$ if $y$ is on the invariant surface $Z$.
\end{lemma}
\begin{proof}
Since any $0$-cycle is rationally equivalent to some $0$-cycle supported on an open subscheme, we may assume that $y\in \tY^\vee$ is a general point such that $y\notin Z$. We recall the description of the map $$F(X)\xrightarrow{\alpha} \tY^\vee \rightarrow Y^\vee,$$ in \cite[Proposition 4.9]{iliev2011fano}.  Then the conics in $\alpha^{-1}(y)$ are identical to  conics in a linear system in $S_{V_4}$ for some $V_4$. The involution $\iota$ sends conics in $S_{V_4}$ to their residual conics.

It suffices to show that for a general $V_4$ and any $c, c'$ residual in  $S_{V_4}$, the sums $c+c'$ are constant in $\CH_1(X)$, which follows from the fact that the hyperplane sections in such $S_{V_4}$'s are parameterized by a $\PP^3$-bundle over $\PP(V_5^{\vee})$.
\end{proof}

\subsection{Relations in \texorpdfstring{$\CH^*(F)$}{F} of incidences}
We introduce an incidence $I\in \CH^2(F\times F)$ and compute $I^2\in \CH^4(F\times F)$, where $F=F(X)$. Denote by $c_s$ the corresponding conic to $s\in F$.
The morphisms are the same as in diagram \eqref{notation}.
\begin{definition}
\begin{enumerate}
    \item The incidence subscheme $I\subset F\times F$ is defined to be 
     \begin{equation*}
     I:=\{(s,t)\in F\times F~|~c_s\cap c_t\neq \emptyset\},
     \end{equation*}
    which is endowed with the reduced closed subscheme structure and of dimension $8$.
    \item $D_c= p(q^{-1}(c))\subset F$ is the locus in $F$ consisting of conics intersecting the given conic $c$ and $D_c'=q^{-1}(c)$.
    
    \item $\Sigma_1\subset F\times F$ is defined to be
    \begin{equation*}
        \Sigma_1:=\{(s,t)\in F\times F~|~c_s ~\text{and}~ c_t~ \text{are union of two lines and have a common component}\},
    \end{equation*}
     which is of dimension $5$.
    \item $\Sigma_2\subset F\times F$ consists of $(s,t)\in F\times F$, where neither $c_s$ nor $c_t$ is of $\tau$-type.
    \item $W\subset F\times F$ is defined to be 
    \begin{equation*}
        W:=\{(s,t)~|~\exists  V_4,c_s~\text{and}~c_t ~\text{are residual in}~S_{V_4}\}.
    \end{equation*}
    For general $s,t$, that means the two conics meet at two points. Since a conic has a pencil of residual conics in $S_{V_4}$, we have $\text{dim}~ W=6$.
\end{enumerate}
\end{definition}

\begin{lemma}\label{plane}
For a general GM fourfold $X$, two different conics $c_1$ and $c_2$ span different planes. If $c$ is of $\tau$-type, then $\langle c\rangle\cap \Gr(2,V_5)=c$.
\end{lemma}
\begin{proof}
A general $X$ contains no planes. Since $X$ is the intersection of quadrics, $\langle c\rangle$ is not contained in some quadric $Q_0\supset X$. Hence, $\langle c\rangle\cap X\subset Q_0\cap \langle c\rangle=c$, which means that $\langle c\rangle$ determines $c$. 

The second statement follows from the fact that the plane spanned by a $\tau$-conic is not contained in $\Gr(2,V_5)$ and $\Gr(2,V_5)$ is also the intersection of quadrics.
\end{proof}
According to Lemma \ref{plane}, we obtain  an injective morphism  
\begin{align}
    F\rightarrow \Gr(3,10), 
\end{align} sending $t\in F$ to $\langle c_t\rangle\subset \PP(\bigwedge^2V_5)$. Let $\PP=\PP(E)$ be the universal plane of $\Gr(3,10)$ restricted to $F$, we have the following diagram:
\begin{equation}
    \xymatrix{
   P \ar[r]^q \ar[rd]^i \ar[rdd]^p & X \ar[rd]             \\
  &\PP\ar[d]^{p'}  \ar[r]^{q'} & \PP(\bigwedge^2V_5)         \\
  & F                      }.
\end{equation}

\begin{lemma}\label{L2.3}
The evaluation morphism $q$ is flat for a general $X$. In particular, the class $I=(p\times p)_*(q\times q)^*\Delta_X\subset \CH^2(F\times F)$. Equivalently, the correspondence $I=$ $^{t}P \circ P$.
\end{lemma}
\begin{proof}
The universal conic $P$ is Cohen-Macaulay since it is a divisor in a smooth variety $\PP$, hence it suffices to show that the dimensions of fibers of $q$ are constant.
Suppose on the contrary that there is a point $x\in X$ such that $q^{-1}(x)$ has a component $K$ of dimension $\geq 3$. We claim that
\begin{claim}
There is a smooth conic $c$ passing through $x$ such that $(c,x)\in P$ belongs to $K$ and $c$ is of $\tau$-type or $\rho$-type.
\end{claim}
\begin{claimproof}
Assume that $x=V_2$. It suffices to show that the dimension of the locus of $\sigma$-conics and singular conics passing through $x$ is less than $3$.

In fact, the $\sigma$-conics are parameterized by the planes $$\PP(V_1\bigwedge V_4)\subset H,$$ where $V_1\subset V_4$. When $V_1\neq W_1$, the locus of $\sigma$-conics passing through $x$ are equal to the image of $\PP(V_2)$ in $Bl_{[W_1]}(\PP(V_5))$, which is $1$-dimensional. When $W_1= V_1$, a $\sigma$-conic  passing through $x$ corresponds to a plane $\langle c\rangle=\PP(V_1\bigwedge V_4)$ with $V_2\subset V_4\subset V_5$, and there is a $2$-dimensional choice of $V_4$.  Now we turn to singular conics. 
Let \(p:\mathcal L\to F_1(X)\) and \(e:\mathcal L\to X\) be the universal family of lines, and set
\[
B=\{y\in X:\dim e^{-1}(y)=1\}.
\]
For general \(X\), \(F_1(X)\) is a smooth threefold, \(e\) is generically finite, and all its fibers have dimension at most one.
We first prove that \(\dim B\leq1\). For every line \(\ell\),
\[
h^0(N_{\ell/X})=3,\qquad \det N_{\ell/X}\simeq\mathcal O_\ell.
\]
The determinant of the normal evaluation map is therefore either zero or nowhere zero. Thus, set-theoretically, the critical locus of \(e\) is \(p^{-1}(D)\) for some \(D\subset F_1(X)\) of dimension at most two. If a surface \(S\subset B\) existed, a three-dimensional component of \(e^{-1}(S)\) dominating \(S\) would lie in this critical locus and hence equal \(p^{-1}(D_0)\) for a surface \(D_0\subset D\). Consequently, \(S\) would contain a two-dimensional family of lines, forcing it to be a plane: the lines through a general smooth point fill its embedded tangent plane. This contradicts the absence of planes in general \(X\).
Now fix \(x\in X\). If \(\ell\not\subset B\), then \(\ell\cap B\) is finite, so the family of lines meeting \(\ell\) has dimension at most one. Only finitely many lines are contained in \(B\), and for each of them the corresponding family has dimension at most two. Since the lines through \(x\) form a family of dimension at most one, reduced singular conics \(\ell\cup\ell'\) through \(x\) form a family of dimension at most $2$. The double-line locus is two-dimensional by \cite[Lemma 3.8]{iliev2011fano}.
Therefore the claim follows. 
\end{claimproof}

The fiber $p^{-1}(c)=c\subset P$ can be identified with $c\subset X$ via $q$ and $$N_{c/P}=\mathrm{H}^0(N_{c/X})\otimes \cO_c.$$ Considering the differential of $q$ restricted to $c\subset P$, we obtain the following diagram, see also \cite[section 2.3]{kuznetsov2018hilbert}:
\begin{equation}\label{differential}
    \xymatrix{
    &0\ar[r] &Tc\ar[r]\ar[d]^*[@180]{\cong} &TP|_c \ar[r]\ar[d]^{dq|_c} &\mathrm{H}^0(N_{c/X})\otimes \cO_c\ar[r]\ar[d]^{ev} &0 \\
    &0\ar[r] &Tc\ar[r] &TX|_c \ar[r] &N_{c/X}\ar[r] &0
    }.
\end{equation}
The normal bundle $N_{c/X}=\cO_c\oplus\cO_c(1)\oplus\cO_c(1)$ for a smooth conic of $\tau$-type or $\rho$-type. Therefore, 
\begin{equation*}
    T_{(c,x)}q^{-1}(x)=\{s\in \mathrm{H}^0(N_{c/X})~|~s_x=0\},
\end{equation*}
 which is $2$-dimensional, contradicting to the fact that $\mathrm{dim}~K\geq 3$.

\end{proof}

\begin{lemma} \label{sec}
Let $c_1$ and $c_2$ be two different conics on a general $X$ and assume that at least one of them is of $\tau$-type. If $c_1$ and $c_2$ have no common component and $\langle c_1\rangle\cap\langle c_2\rangle$ is a line $l$, there exists a $4$-dimensional subspace $V_4\subset V_5$, such that $\langle c_1\rangle\cup \langle c_2\rangle\subset \PP(\bigwedge^2V_4)$. In particular, $\iota(\alpha([c_1]))=\alpha([c_2])$.
\end{lemma}
\begin{proof}
A general $X$ contains no planes and hence $l\not\subset X$.
The conic $c_i$ is contained in $\PP(\bigwedge^2V_{4i})$ for some $4$-dimensional vector space $V_{4i}$. If $V_{41}\neq V_{42}$, then $V_3:=V_{41}\cap V_{42}$ is $3$-dimensional and $$l\subset \PP(\bigwedge^2V_3)\subset \Gr(2,V_5).$$ In particular, $\langle c_i\rangle\cap \Gr(2,V_5)\neq c_i$. Therefore
by Lemma \ref{plane}, we have $\langle c_i\rangle\subset \Gr(2,V_5)$ and it means that the two conics are of $\sigma$-type or $\rho$-type, contradicting to the fact that one of the two conics is of $\tau$-type. 

Hence $V_{41}=V_{42}=V_4$ and $c_1 , c_2$ are contained in  $S_{V_4}$. Then $\iota(\alpha(c_1))=\alpha(c_2)$ follows from that $\iota$ sends a conic to its residual conic in $S_{V_4}$. 
\end{proof}

We define $\tilde{I}\subset P\times P$ and $I'\subset \PP\times \PP$ to be the incidence correspondences, i.e. 
    $$\tilde{I}=(q\times q)^{-1}\Delta_X {\rm ~and~} I'=(q'\times q')^{-1}\Delta_{\PP(\bigwedge^2V_5)}.$$ 
 Let $I_0=I\setminus (\Sigma_1\cup\Sigma_2\cup W)$, $\tilde{I_0}=(p\times p)^{-1}I_0$ and $I_0'=(p'\times p')^{-1}I_0$. 
Then Lemma \ref{sec} says that there is a section from $I_0$ to  $I'_0$, sending two intersecting conics to their common point. 
\begin{lemma}\label{L2.5}
The inclusion $I_0\subset F\times F\setminus (\Sigma_1\cup\Sigma_2\cup W)$ is a regular embedding.
\end{lemma}
\begin{proof}
The inclusion $I'\subset \PP\times\PP$ is a regular embedding since it is obtained from $\Delta_{\PP(\bigwedge^2V_5)}\subset \PP(\bigwedge^2V_5)\times \PP(\bigwedge^2V_5)$ via base change. By the discussion above, there exists a section from $I_0$ to $I'_0$. We apply \cite[B.7.5]{fulton2016intersection} and obtain that $I_0\subset F\times F \setminus (\Sigma_1\cup\Sigma_2\cup W)$ is a regular embedding. 
\end{proof}

Now we can compute the self-intersection $I^2$, the analogous case of cubic fourfolds can be found in \cite[Proposition 2.3]{voisin4chow} or \cite[Appendix A.4]{fu2019generalized}.
\begin{proposition}\label{formula}
The self-intersection $$I^2=aW+I\cdot A+B+C\in \CH^4(F\times F)$$ for some constant $a$.  Here, $A$ and $2B$ are contained in $pr_1^*\CH^*(F)\cdot pr_2^*\CH^*(F)\subset \CH^*(F\times F)$, where $pr_i$ is the  projection and $C$ is supported on $\Sigma_2\subset F\times F$.
\end{proposition}
\begin{proof}
By Lemma \ref{L2.5}, $I_0\subset F\times F\setminus (\Sigma_1\cup\Sigma_2\cup W)$ is a regular embedding and the restriction $p_0$ of $p$ from $\tilde{I_0}\subset P\times P$ to $I_0\subset F\times F$ is an isomorphism. Then, $$I_0^2=c_2(N_{I_0/F\times F})$$ and 
\begin{equation*}
    0\rightarrow pr_1^*T_{P/F}\oplus pr_2^*T_{P/F} \rightarrow N_{\tilde{I}_0/P\times P}\rightarrow p_0^*(N_{I_0/F\times F})\rightarrow 0.
\end{equation*}
Let $q_0$ be the map $q_0: \tilde{I_0}\rightarrow \Delta_X=X$ which maps two conics in $\tilde{I_0}$ to their common point. Since $\tilde{I}=(q\times q)^{-1}\Delta_X$, we have 
\begin{equation*}
    N_{\tilde{I}_0/P\times P}= q_0^*(T_X)= pr_1^*(T_X)\cdot \tI_0=pr_2^*(T_X)\cdot \tI_0
\end{equation*}
and
\begin{equation}\label{normalbundle}
    c(p_0^*(N_{I_0/F\times F}))=\frac{q_0^*c(T_X)}{pr_1^*c(T_{P/F})\cdot pr_2^*c(T_{P/F})}.
\end{equation}
Let $K=\sigma_{1,1}|_X$ be the class of $S_{V_4}$ and $H=\sigma_1|_X$ is the hyperplane class.  We claim that $c_2(p_0^*(N_{I_0/F\times F}))$ is a degree 2 polynomial with variables in $pr_1^*\CH^*(F)\cdot pr_2^*\CH^*(F)$ and $pr_1^*\langle H,K\rangle\cdot pr_2^*\langle H,K\rangle$ restricting to $\tI_0$.  The following  computations is standard:
\begin{enumerate}
    \item  The Chern classes of $X$  
           \begin{equation*} 
              c(T_X)=\frac{c(T_{\Gr(2,V_5)})|_X}{(1+H)(1+2H)}
          ,\end{equation*}
          where $c(T_{\Gr(2,V_5)})$ is a polynomial in Schubert cycles. Therefore, the degree $1$ and $2$ parts of $q_0^*c(T_X)$ are polynomials in  $H$ and $\sigma_{1,1}$  via the two projections and restriction to $\tI_0$.
    \item The Chern classes of the relative tangent bundle $T_{P|F}$
           \begin{equation*}
              c(T_{P|F})=\frac{c(T_{\PP/F})|_P}{c(N/\PP)}= \frac{c(p^*(E)\otimes\cO_{\PP}(1))|_P}{c(\cO_P(P))}   
          ,\end{equation*}    
          where $c_1(\cO_{\PP}(1))|_X=H$ and $P=2c_1(\cO_{\PP}(1))+p^*D\in \CH^1(\PP)$ for some $D\in \CH^1(F)$. So the degree $1$ and degree $2$ parts are polynomials in $H$ and $\CH^*(F)$.    
\end{enumerate}
The claim follows from the computations and \eqref{normalbundle}. 
Next, the terms in $p_0^*c_2(N_{I_0/F\times F})$ are of the following three types:
\begin{enumerate}
    \item $D\cdot \tI_0$, where $D\in pr_1^*\CH^*(F)\cdot pr_2^*\CH^*(F)$.
    \item $(q\times q)^*pr_i^*H\cdot D\cdot \tI_0$ for $i=1$ or $2$, where $D$ is of form $(q\times q)^*pr_j^*H$ or $(p\times p)^*pr_j^*D_1$ for some $D_1\in \CH^1(F)$.
    \item $(q\times q)^*pr_i^*K\cdot \tI_0$ for $i=1$ or $2$.
\end{enumerate}
Applying $p_{0*}$ to $p_0^*c_2(N_{I_0/F\times F})$, the terms of the first type are of the form $I\cdot A$ by projection formula. For the other two types, we use $\tI=(q\times q)^*\Delta_X$ and it's enough to show that $pr_i^*H\cdot pr_j^*H\cdot\Delta_X$ and $pr_i^*K\cdot\Delta_X$ are in $pr_1^*\CH^*(X)\cdot pr_2^*\CH^*(X)$. Denote by $M=\Gr(2,V_5)\cap H$  the Grassmannian hull of $X$ and consider
\begin{equation*}
    j_1:X\times X\hookrightarrow M\times X.
\end{equation*}
We know that $M$ is a linear variety (Schubert varieties are linear in the sense of \cite{totaro2014chow}, since they are stratified by Schubert cells) and therefore
by \cite[Proposition 1]{totaro2014chow}, $$\CH^*(M\times X)=pr_1^*\CH^*(M)\cdot pr_2^*\CH^*(X).$$ Hence we have 
\begin{equation*}
    j_1^*j_{1*}\Delta_X\in pr_1^*\CH^*(X)\cdot pr_2^*\CH^*(X)
\end{equation*}
Since $X$ is of class $2H$ in $M$, we have $ j_1^*j_{1*}\Delta_X=2pr_1^*H\cdot \Delta_X$. Hence $pr_1^*H\cdot\Delta_X\in pr_1^*\CH^*(X)\cdot pr_2^*\CH^*(X)$.

For the third type, we recall that the class $K\in \CH^2(X)$ is represented by a Del Pezzo surface $S=S_{V_4}$ and  $pr_i^*K\cdot\Delta_X$ is the push-forward of the diagonal class $\Delta_S$ in $\CH^2(S\times S)$.  Then by the decomposition of diagonal, we see that $$\Delta_S=\alpha\times S+\Gamma,$$ where $\alpha\in \CH_0(S)$ and $\Gamma$ is supported on $S\times V$ for some proper subvariety $V\subset S$. If $\dim V=0$, $\Gamma$ is of the form $S\times \beta$ for some $\beta\in \CH_0(X)$. If $\dim V=1$, then $\Gamma\in \CH^1(V\times X)$. We see that $\CH^1(V\times X)=pr_1^*\CH^1(V)\cdot pr_2^*\CH^1(S)$ since $\rH^1(S,\cO_S)=0$. In both cases, we obtain: 
$$\Delta_S\in pr_1^*\CH^*(S)\cdot pr_2^*\CH^*(S),$$  which implies that $pr_i^*K\cdot\Delta_X$ are in $pr_1^*\CH^*(X)\cdot pr_2^*\CH^*(X)$.

Finally, we get the conclusion by the localization exact sequence for Chow groups.

\end{proof}
Since $A$ and $2B$ are sums of exterior products, we have both $A_*\CH_0(F)_{\rm hom}=0$ and $B_*\CH_0(F)_{\rm hom}=0$. The latter is due to the divisibility of $\CH_0(F)_{\rm hom}.$
\subsection{\texorpdfstring{$Z$}{Y} is a constant cycle surface }

Now we can prove that $Z$ is a constant cycle surface on $\tY^\vee$, and a point on $Z$ represents the distinguished 0-cycle $o$ that appeared in \eqref{weak splitting}. First, we notice that $$\Phi([c])=I_*([c])=D_c$$ since $q$ is flat, where we recall that $\Phi=p_*q^*$. Hence the class of $D_c$ depends only on the class $[c]\in \CH_1(X)$. 

For a general $V_4$, $S_{V_4}$ is isomorphic to the blow-up of $\PP^2$ at five points. Let $H$ be the pullback of the hyperplane class on $\PP^2$ and $E_i$, $i=1,\dots 5$, be the five exceptional curves. There are $10$ pencils of conics on $S_{V_4}$, which are $H-E_i$ and $2H-\sum_{j\neq i}E_j$. Denote by $L_i,~i=1,\dots 10$ the corresponding lines on $F$ and let $L_{i+10}=L_{i}$. Recall that $L_i$'s are contracted to a point by $\alpha$ and $\alpha(L_i)=\iota(\alpha(L_{i+5}))$. By a straightforward computation in $S_{V_4}$, we have \begin{equation}\label{conicindelpezzo}
     c_i\cdot c_j=1 ~{\rm for}~ j\neq i,i+5;  c_i\cdot c_{i+5}=2 ~{\rm and}~ c_i^2=0 ~{\rm for ~any}~ i.
\end{equation} 
Here, $c_i$ is a conic in $L_i$.
That is, conics in $L_{i+5}$ meet $c_i$ at two points and  conics in $L_j,j\neq i,i+5$ meet $c_i$ at one point.    
We have the following computation:
\begin{lemma}\label{KeyLemma}
Let $c_i$ be a conic on a general $S_{V_4}$ and $c_i\in L_i$, then 
\begin{equation}\label{error}
    D_{c_i}^2=\sum_{j\neq i,j\neq i+5} L_j+4L_{i+5}+kL_{\sigma} \in \CH_1(F),
\end{equation}
where $L_{\sigma}\subset F$ is a curve contained in the locus of $\sigma$-conics.
\end{lemma}
\begin{remark}
   As pointed out by Carl Mazzanti, Formula \eqref{error}
    is missing a contribution from certain reducible conics.
    The correction and the necessary modifications to the
    subsequent proofs are given in his appendix.
    All the main results of this paper remain unchanged. 
\end{remark}

\begin{proof}
We may assume that $S_{V_4}$ is general so that $W_1\not\subset V_4$ and $S_{V_4}$ contains no $\sigma$-conics. Take another conic $c_0\in |c_i|$ and thus $D_{c_0}=D_{c_i}$ in $\CH^2(F)$. The intersection $D_{c_i}\cap D_{c_0}$ consists of conics $c$ which meet both $c_i$ and $c_0$. We distinguish between the two cases whether $c$ is contained in $S_{V_4}$ or not. 

If $c\in D_{c_i}\cap D_{c_0}$ and $c\subset S_{V_4}$,  it is contained in another linear system of conics, supported on $L_j$, for some $j\neq i$. Conversely, $L_j\subset {\rm Supp}~ D_{c_i}\cap D_{c_0}$ for $j\neq i$ by \eqref{conicindelpezzo}. 

If $c\in D_{c_i}\cap D_{c_0}$ and $c\not\subset S_{V_4}$, we first show that $c$ must be a $\sigma$-conic.
Let $x\in c\cap c_0$ , $y\in c\cap c_i$ and $V_x,V_y$ be the corresponding 2-dimensional subspaces in $V_5$.  We have $V_x\cup V_y$ is a three-dimensional space $V_3$, otherwise $c$ would be contained in $S_{V_4}$.
It follows that $V_x\cap V_y$ is a one-dimensional space $V_1$. Therefore $$\PP(V_1\bigwedge V_3)\subset \Gr(2,V_5)\cap \langle c\rangle.$$ By Lemma \ref{plane}, $c$ is a $\sigma$-conic or a $\rho$-conic. Since  $W_1\not\subset V_4$, we see that $W_1\not\subset V_3$ and $c$ is not a $\rho$-conic. Therefore $c$ is a $\sigma$-conic.

Recall that $\sigma$-conics are parameterized by $Bl_{[W_1]}(\PP(V_5))$. 
Let $S_{c_0}$ and $S_{c_i}$ be the surfaces swept out by lines parameterized by $c_0$ and $c_i$ in $\PP(V_4)$. There is a natural morphism 
$$S_{c_0}\cap S_{c_i}\rightarrow Bl_{[W_1]}(\PP(V_5)),$$
and the image is exactly the locus of $\sigma$-conics meeting $c_i$ and $c_0$ by the previous discussion.
The intersection $S_{c_0}\cap S_{c_i}$ is clearly one dimensional.  Therefore the conics in  $D_{c_i}\cap D_{c_0}$ not contained in $S_{V_4}$ are parameterized by a curve $L_{\sigma}\subset F$.

Thus $D_{c_i}$ and $D_{c_0}$ intersect dimensionally transversely and it remains to compute the multiplicity. We consider the differential $dq$ as in \eqref{differential}.
Let $c$ be a general smooth conic in $D_{c_i}\cap D_{c_0}$ and $c\in L_j,j\neq i, i+5$.  Then $c$ meets $c_i$ (resp.~$c_0$) at one point $x$ (resp.~$y$). We have $T_x c\subset T_xS_{V_4}$, hence the image of $T_x c$ in $N_{c/X}$ is $N_{c/S_{V_4},x}$.
Then by \eqref{differential},  we have $$T_{(c,x)}D_{c_i}'=dq^{-1}(T_xc_i)=\mathrm{H}^0(N_{c/X}\otimes I_x)\oplus \mathrm{H}^0(N_{c/S_{V_4},x}),$$ which is $3$-dimensional. Thus  $D_{c_i}'$ is smooth at $(c,x)$. Since $c_i$ and $c$ meet at a single point, $D_{c_i}'$ maps isomorphically to $D_{c_i}$ via $p$ around $c$. Hence,
\begin{equation*}
    T_{c}D_{c_i}=\mathrm{H}^0(N_{c/X}\otimes I_x)\oplus \mathrm{H}^0(N_{c/S_{V_4}})\subset T_{c}F.
\end{equation*}
Similarly  for $c_0$, we obtain 
\begin{equation*}
    T_{c}D_{c_0}=\mathrm{H}^0(N_{c/X}\otimes I_y)\oplus \mathrm{H}^0(N_{c/S_{V_4}})\subset T_{c}F.
\end{equation*}
Since $\mathrm{H}^0(N_{c/X}\otimes I_x\otimes I_y)=0$, we have 
\begin{equation*}
    T_{c}D_{c_i}\cap T_{c}D_{c_0}=\mathrm{H}^0(N_{c/S_{V_4}})=T_{c}L_j.
\end{equation*}
That means  $D_{c_i}$ and $D_{c_0}$ meet transversely at $c$, hence the multiplicity of $L_j$ is $1$ for $j\neq i,i+5$.

For a general conic $c\in D_{c_i}\cap D_{c_0}\cap L_{i+5}$, $c$ meet $c_i$ and $c_0$ at two points. Let 
$c\cap c_i=x_1,x_2$ and $c\cap c_0=y_1,y_2$. Then $p: D_{c_i}'\rightarrow D_{c_i}$ (resp. $D_{c_0}'\rightarrow D_{c_0}$) maps the two points $(c,x_1)$ and $(c,x_2)$ (resp. $(c,y_1)$ and $(c,y_2)$ ) to $c$. We obtain that the tangent cones 
\begin{equation}\label{cone}
\begin{split}
        C_{c}D_{c_i} &= \mathrm{H}^0(N_{c/X}\otimes I_{x_1})\oplus \mathrm{H}^0(N_{c/S_{V_4}})\cup \mathrm{H}^0(N_{c/X}\otimes I_{x_2})\oplus \mathrm{H}^0(N_{c/S_{V_4}})\\
        C_{c}D_{c_0} &= \mathrm{H}^0(N_{c/X}\otimes I_{y_1})\oplus \mathrm{H}^0(N_{c/S_{V_4}})\cup \mathrm{H}^0(N_{c/X}\otimes I_{y_2})\oplus \mathrm{H}^0(N_{c/S_{V_4}}).
\end{split}
\end{equation}
Take a general hypersurface $G\subset F$ containing $c$ such that $G$ meets $L_{i+5}$ transversely at $c$ and then the multiplicities $${\rm mult}_{c} D_{c_i}|_G= {\rm mult}_{c} D_{c_0}|_G = 2.$$ 
Combining the fact that the direction of $\mathrm{H}^0(N_{c/S_{V_4}})$ is in the fiber of $\alpha$, we obtain $$C_{c}D_{c_i}|_G\cap C_{c}D_{c_0}|_G=0$$ by \eqref{cone} and the description of the normal bundle of conics.  By \cite[Proposition 1.29]{eisenbud20163264}, the intersection multiplicity of $D_{c_i}|_G$ and $D_{c_i}|_G$ at $c$ is $4$.  It follows that the multiplicity of the intersection of $D_{c_0}$ and $D_{c_i}$ at $L_{i+5}$ is $4$.
\end{proof}
Apply this Lemma to every $c_i\subset S_{V_4}$ and by a linear combination, we have:
\begin{equation}\label{inverseformula}
    6(D_{c_i}^2+D_{c_{i+5}}^2)-\sum_{j=1}^{10}D_{c_j}^2=12(L_i+L_{i+5})+L'  ,
\end{equation}
where $L'$ is a 1-cycle on $F$ consisting of $\sigma$-conics.

\begin{lemma}\label{nonzero}
The constant $a$ in Proposition \ref{formula} is nonzero.  If $n[c_1]\sim n[c_2]$ in $\CH_1(X)$ for some nonzero integer $n$, then $[c_1]\sim[c_2]$ in $\CH_0(\tY^\vee)$.
\end{lemma}
\begin{proof}
We assume on the contrary that $a=0$. Then $$I^2=I\cdot A+B+C\in \CH^4(F\times F),$$ 
where $A,B,C$ are as in Proposition \ref{formula}. 
A correspondence $\Gamma$ of codimension larger than $0$ in $pr_1^*\CH^*(F)\cdot pr_2^*\CH^*(F)$  is of the form $[Z_1\times Z_2]$, where either $Z_1$ or $Z_2$ is a proper subvariety of $F$. In both cases, $\Gamma$ induces a constant map  $\CH_0(F)\rightarrow \CH^*(F)$. Further, by Lemma \ref{L2.3}, $I=$ $^{t}P \circ P$ and hence $I_*: \CH_0(F)\rightarrow \CH_3(F)$ factor through $\CH_1(X)$. Together with the torsion-freeness of $\CH_0(F)$, it implies that if $a=0$, then for $\xi,\zeta \in \CH_0(F)$, 
\begin{equation}\label{contrary}
    I^2_*(\xi)=I^2_*(\zeta)
\end{equation}
as long as $q_*p^*(n\xi)=q_*p^*(n\zeta)\in \CH_1(X)$, where $p,q$ are the projections from the universal conic to $F$ and $X$. 

According to \cite[Lemma 17.3]{shen2016fourier}, we know that $I^2_*([c])=D_c^2$. Then \eqref{contrary} together with Lemma \ref{2.1} imply that $D_c^2+D_{\iota(c)}^2$ are constant in $\CH_1(F)$ when varying $c\in F$. 
We want to deduce a contradiction between the formula \eqref{inverseformula} and the constantness of $D_c^2+D_{\iota(c)}^2$.

Take a very ample divisor $G\subset F$.  Intersecting both sides of formula \eqref{inverseformula} with $G$ and pushing forward to $\tY^\vee$, then the constantness of $D_c^2+D_{\iota(c)}^2$ implies that $([c]+\iota([c]))$ is constant in $\CH_0(\tY^\vee)$, i.e., points on $$\Delta'=\{(y,\iota(y))~|~y\in \tY^\vee\}\subset \tY^\vee\times \tY^\vee$$ are constant in $\CH_0(\tY^\vee)$. 

Then for any power $\sigma^l$ of the two-form $\sigma$ on $\tY^\vee$, we have $(pr_1^*(\sigma^l)+pr_2^*(\sigma^l))|_{\Delta'}=0$ by Mumford's theorem, see  \cite[Proposition 10.24]{voisin2003hodge}. Then we obtain $$2pr_1^*(\sigma^2)|_{\Delta'}=0,$$ but it is impossible since $pr_1$ is an isomorphism between $\Delta'$ and $\tY^\vee$.    Therefore, $a$ is nonzero.

For the second part, assume that $n[c_1]\sim n[c_2]$ and again take a very ample divisor $G\subset F$. Let $m$ be the intersection number of $G$ and a general fiber of $\alpha:F\rightarrow \tY^\vee$. 
Then we have $n^2I^2_*([c_1])=n^2D_{c_1}^2=n^2D_{c_2}^2=n^2I^2_*([c_2])$ again by \cite[Lemma 17.3]{shen2016fourier} and $nI_*([c_1])=nI_*([c_2])$. Hence due to Proposition \ref{formula} , 
\begin{equation*}
    \begin{split}
        0=&G\cdot n^2(I^2_*([c_1])-I^2_*([c_2])) \\=&G\cdot n^2(aW+I\cdot A+B+C)([c_1]-[c_2]) \\=&G\cdot an^2W([c_1]-[c_2])\\=&amn^2\cdot (\iota([c_1])-\iota([c_2])),
    \end{split}
\end{equation*}
here we view the cycles in $\CH^4(F\times F)$ as morphisms from $\CH_0(F)$ to $\CH_1(F)$. Then $$\iota([c_1])\sim \iota([c_2])$$ again by the torsion-freeness of $\CH_0(\tY^\vee)$ and the fact that $a\neq 0$. Since $\iota$ is an involution, the result follows.

\end{proof}

Combining Lemma \ref{2.1} and Lemma \ref{nonzero}, we can get the following :
\begin{theorem}\label{constant}
The invariant locus $Z$ of the involution is a constant cycle surface on $\tY^\vee$.
\end{theorem}

Thus by \eqref{weak splitting}, $h^2\cdot Z=40o$ in $\CH_0(\tY^\vee)$. That means $o$ is represented by a point on the constant cycle surface $Z$. We show that $o$ is in fact the class represented by a point on any constant cycle surface:
\begin{proposition}\label{o}
For a very general $\tY^\vee$, if $Z'\subset \tY^\vee$ is another constant cycle surface, then for any point $z'\in Z'$, $z'$ is rationally equivalent to $o$.
\end{proposition}

\begin{proof}
It suffices to show that for a very general $\tY^\vee$, any two surfaces on $\tY^\vee$ intersect. We obtain this by showing that
every surface on $\tY^\vee$ is strictly nef, i.e., having a positive intersection number with any non-zero effective $2$-cycle.

Let $N_2(\tY^\vee)$ be the space of  $2$-cycles modulo numerical equivalence with $\mathbb{R}$-coefficients, which is a $2$-dimensional vector space spanned by $h^2$ and $Z$.
As in \cite{ottem2015nef}, the Lagrangian surface $Z$ is in the boundary of the effective cone $\overline{Eff}_2(\tY^\vee)$, and $c_2(\tY^\vee)$ is not contained in the interior of the effective cone. Thus 
$$\overline{Eff}_2(\tY^\vee) \subset \langle \mathbb{R}_{\geq 0}(Z), \mathbb{R}_{\geq 0}(c_2)\rangle.$$
Then by \eqref{intersection number}, we have 
\begin{equation*}
\begin{split}
      \langle \mathbb{R}_{\geq 0}(Z), \mathbb{R}_{\geq 0}(c_2)\rangle&\subset\langle \mathbb{R}_{\geq 0}(24h^2-5Z), \mathbb{R}_{\geq 0}(-2h^2+5Z)\rangle\\
      &= (\langle \mathbb{R}_{\geq 0}(Z), \mathbb{R}_{\geq 0}(c_2)\rangle)^{\vee}\\
      &\subset Nef_2(\tY^\vee).
 \end{split}
 \end{equation*}
The first inclusion is such that the cones have no boundary in common by a straightforward computation. Hence every surface on $\tY^\vee$ is strictly nef.
 \end{proof}
 
\section{The filtration on \texorpdfstring{$\text{CH}_1(X)$}{CH\_1(X)}}

\subsection{Basics of the filtration}
A uniruled divisor on a $2n$-dimensional IHS variety $M$ is a divisor $D$ which admits a rational map to a $(2n-2)$-dimensional variety $B$:
\begin{equation*}
    \begin{tikzcd}
D \arrow[r, hook] \arrow[d, "q"', dashed] & M, \\
B                                         &  
\end{tikzcd}
\end{equation*}
and the general fibers of $q$ are rational curves. By \cite{charles2019families}, there exists a uniruled divisor on an IHS variety of $K3^{[n]}$-type with $n\leq 7$, in particular, on the dual double EPW sextic $\tY^\vee$.

The Beauville–Voisin filtration on $\tY^\vee$ is determined by uniruled divisors and constant cycle surfaces. Points on uniruled divisors serve as the $1$st piece and points on constant cycle surfaces serve as the $0$th piece. According to \cite[Lemma 1.1]{shen2020categories} and Proposition \ref{o}, the filtration does not depend on the choice of a uniruled divisor and a constant cycle surface on a very general dual double EPW sextic. 

We define a filtration on $\CH_1(X)$ for a very general GM fourfold induced by the Beauville–Voisin filtration on $\tY^\vee$. In the rest of the paper, we study some properties of the filtration. We first define a special class:
\begin{lemma}\label{lem:theta}
   Let $X$ be a very general GM fourfold, and let
$o\in\CH_0(\tY^\vee)$ be the class of a point on $Z$. Then, there exists a line $l\subset X$ such that
$$
2[l]=\Psi(o).
$$
The class $\theta:=[l]$ is independent of the choice of $l$ and of degree $1$.
\end{lemma}
\begin{proof}
By \cite[Lemma 3.8]{iliev2011fano}, the double-line locus is two-dimensional and generically of $\tau$-type.
By \cite[Proposition 4.9]{iliev2011fano}, the corresponding $\alpha$-fibers come from rulings on $S_{V_4}$.
This surface is integral since $X$ contains no plane. Hence generic smoothness implies that each such ruling contains only finitely many double lines. Therefore, the image of the double-line locus is two-dimensional, which meets $Z$ by Proposition \ref{o}. Thus there exists a double line $c\subset X$, supported on a line $l$, such that
$\alpha(c)\in Z$. Since $Z$ is a constant cycle surface,
$$
\Psi(o)=\Psi([\alpha(c)])=[c]=2[l].
$$
If we also have $2[l']=\Psi(o)$, then we obtain $2([l]-[l'])=0$. Hence by the torsion-freeness, we have $[l]=[l'].$

\end{proof}
\begin{definition}
The $i$th piece of the increasing filtration $\rS_\bullet X\subset \CH_1(X)$  consists of $z\in \CH_1(X)$, such that (i.e., the class of $az$ can be represented by a sum of multiples of $i$ conics in a uniruled divisor and $\theta$) 
\begin{equation*}
    az=a_1c_1+a_2c_2\ldots +a_ic_i+a_0\theta
\end{equation*}
for some integer $a\neq 0,a_j$, where $c_j$'s are conics with $\alpha(c_j)$ in a uniruled divisor $D\subset \tY^\vee $. Here, $\theta$ is the class in Lemma \ref{lem:theta}
\end{definition}

The EPW sextics associated with period partners $X$ and $X'$ are the same. We have an immediate corollary according to the definition:
\begin{corollary}\label{partner}
For any $t\in \CH_0(\tY^\vee)$, $\Psi(t)\in \rS_i(X)$ iff $\Psi'_*(t)\in \rS_i(X')$, where $\Psi':\CH_0(\tY^\vee)\rightarrow \CH_1(X')$ is similarly defined as $\Psi$.
\end{corollary}

By the proof of Proposition \ref{ch}, there exists an integer $N$  such that $N\cdot z$ is a linear sum of conics for every $z\in \CH_1(X)$. At the end of the section, we show in Proposition \ref{conic} that every conic is rationally equivalent to a sum of conics in $D$. Hence, 
\begin{equation*}
    \bigcup_{i=0}^{\infty} \rS_i(X)=\CH_1(X).
\end{equation*}

Let $\CH_0(\tY^\vee)^{-}_{\rm hom}$ be the $\iota$-anti-invariant part of $\CH_0(\tY^\vee)_{\rm hom}$ and $\CH_0(\tY^\vee)^{+}_{\rm hom}$ be the $\iota$-invariant part.
There is a decomposition of $\CH_0(\tY^\vee)$:
\begin{equation*}
    \CH_0(\tY^\vee)=\ZZ\cdot o\oplus \CH_0(\tY^\vee)^{-}_{\rm hom}\oplus \CH_0(\tY^\vee)^{+}_{\rm hom}.
\end{equation*}

We show that the filtration on $\CH_1(X)$ can actually be defined on the $\iota$-anti-invariant part $\CH_0(\tY^\vee)^{-}$.

\begin{proposition}\label{decomposition}
The morphism $\Psi:\CH_0(\tY^\vee)\rightarrow \CH_1(X)$  is zero on $\CH_0(\tY^\vee)^{+}_{\rm hom}$ and an isomorphism onto $\CH_1(X)_{\rm hom}$ when restricting to $\CH_0(\tY^\vee)^{-}_{\rm hom}$.
\end{proposition}
\begin{proof}
Due to the divisibility of $\CH_0(\tY^\vee)_{\rm hom}$, the elements in $\CH_0(\tY^\vee)^{+}_{\rm hom}$ are of the form $\sum (t_i+\iota(t_i)-2o)$ and elements in $\CH_0(\tY^\vee)^{-}$ are of the form $\sum t_i-\iota(t_i)$, where $t_i\in\tY^\vee$.
This yields  $\Psi$ is zero on $\CH_0(\tY^\vee)^{+}_{\rm hom}$ by Lemma \ref{2.1}. 

For the second part, there exists an integer $N$, such that for any $z\in \CH_1(X)$, $N\cdot z$ is a linear combination of conics. Then $\CH_1(X)_{\rm hom}$ is generated by conics by the divisibility of $\CH_1(X)_{\rm hom}$.
Therefore, it suffices to show that $\Psi$ is injective on $\CH_0(\tY^\vee)^{-}$. 

Let $G\subset F$ be an ample divisor and $G.L=m$, where $L$ is a general fibre of $\alpha: F\rightarrow \tY^\vee$. Denote by $c_i$ and $c_i'$ the conics representing the class $\Psi(t_i)$ and $\Psi(t_i')$.
If we have the relation $$\Psi(\sum t_i-\iota(t_i))=\Psi(\sum t_i'-\iota(t_i'))\in \CH_1(X),$$  it implies that:
\begin{equation}\label{ker}
    \sum (D_{c_i}-D_{\iota(c_i)})=\sum (D_{c_i'}-D_{\iota(c_i')}).
\end{equation}
By Lemma \ref{KeyLemma}, we have:
\begin{equation*}
     G\cdot (D_{c_i}-D_{\iota(c_i)})\cdot (D_{c_i}+D_{\iota(c_i)})=G\cdot (D_{c_i}^2-D_{\iota(c_i)}^2)=4m(\iota(t_i)-t_i).
\end{equation*}
Here, we identify $\CH_0(F)$ and $\CH_0(\tY^\vee)$ via the isomorphism $\alpha_*$.
By Lemma \ref{2.1}, we know that $D_c+D_{\iota(c)}$ is constant in $\CH^2(F)$,
we denote it by $D_0$. Therefore we obtain
\begin{equation*}
      G\cdot D_0 \cdot \sum (D_{c_i}-D_{\iota(c_i)})=4m\sum (\iota(t_i)-t_i) 
\end{equation*}
and
\begin{equation*}
      G\cdot D_0 \cdot \sum (D_{c_i'}-D_{\iota(c_i')})=4m\sum(\iota(t_i')-t_i').
\end{equation*}
The left-hand sides are equal by \eqref{ker}. Therefore by torsion-freeness we obtain: $$\sum t_i-\iota(t_i)=\sum t_i'-\iota(t_i').$$ Hence $\Psi$ is injective on $\CH_0(\tY^\vee)^{-}$.
\end{proof}

Consequently, we immediately deduce that
\begin{theorem}\label{chowiso}
We have an isomorphism between groups
\begin{equation}
    \ZZ\cdot o\oplus \CH_0(\tY^\vee)^{-}_{\rm hom}\cong \ZZ\cdot \theta \oplus \CH_1(X)_{\rm hom} \cong \CH_1(X).
\end{equation}
In particular, the group $\CH_1(X)$ is torsion-free. 
\end{theorem}
When $\tY^\vee$ is birational to $S^{[2]}$ for some K3 surface $S$,  it has been shown in \cite{laterveer2021zero} that $\CH_0(\tY^\vee)_\QQ$ with $\QQ$-coefficient has a natural decomposition for the Chow group. For 0-cycles, the decomposition is
$$\CH_0(\tY^\vee)_\QQ=\QQ\cdot o\oplus\CH_0(S)_{\rm hom}\oplus\CH_0(\tY^\vee)^+_{\rm hom}.$$
Theorem \ref{chowiso} can be viewed as a generalization of this decomposition, where the group $\CH_1(X)$ takes the place of $\CH_0(S)$.

\subsection{A result for conics}
In this subsection, we prove that the class of a conic on $X$ is in $\rS_2(X)$. The proof is similar to the case of cubic fourfolds by showing any conic is contained in a certain singular cubic threefold.
Here, we need the flexibility to change the GM fourfold to its period partners, guaranteed by Corollary \ref{partner}. 

We consider the following incidence relation $$\Omega:=\{(A,B)~|~\mathrm{dim} (A\cap B)\geq 9\}\subset \LL\GG(\bigwedge^3V_6)\times \LL\GG(\bigwedge^3V_6).$$ Denote the fibre of the projection over $A$ by $\Omega_A$ and $F_y=y\wedge \bigwedge^2V_6\subset \bigwedge^3V_6$. Clearly, we have $$Y_B^{\geq 2}\subset Y_A^{\geq 1}$$ for $(A,B)\in \Omega$. We let 
$$\Sigma=\{A\in \LL\GG(\bigwedge^3V_6)~|~ A {\rm ~contains~ a~ decomposable~ vector}\},$$
and $\Sigma_k$ be the closure of the locus of $A$ that contains exactly $k$ decomposable vectors. 

We can deduce the following lemmas.  
\begin{lemma}
For any $A,B\subset \bigwedge^3V_6$ with $(A,B)\in \Omega$ and $A\neq B$,  we can find  $V_5\subset V_6$, such that the corresponding GM varieties satisfy $X_B\subset X_A$ and are of dimension $3$ and $4$.
\end{lemma}
\begin{proof}
It's enough to take $V_5\in Y_{A^{\vee}}^{=1}\cap Y_{B^{\vee}}^{=2}$ by Theorem \ref{moduli}.
\end{proof}

\begin{lemma}
For a general $B\in \Sigma_8$ and $V_5\in Y_{B^{\vee}}^{=2}$, let $F(X_B)$ be the surface of conics on the corresponding GM threefold $X_B$,  then  \rm{dim} Alb$(F(X_B))=2$.
\end{lemma}
\begin{proof}
For a general choice of $B\in \Sigma_8$ and $V_5$, we may assume that the $8$ decomposable vectors are $$\bigwedge^3V_{31},\ldots \bigwedge^3V_{38}\in B\cap \bigwedge^3V_6$$ and $V_{31}, \ldots V_{38}\not\subset  \bigwedge^3V_5$. Then, by \cite[Proposition 2.24]{debarre2018gushel}, the Grassmannian hull $\Gr(2,V_5)\cap W_B$ is smooth. 
$\bigwedge^3V_{31}$ can be written as $v\wedge v_1\wedge v_2$ with $v\in V_6\setminus V_5$, and $O=v_1\wedge v_2$ is the kernel of $q(v)$, by \cite[Theorem 3.16]{debarre2018gushel}. Then by \cite[Lemma 4.1]{debarre2011nodal}, $X_B$ has a node at $O$. Then we can project $X_B\subset \PP^7$ to $X_O\subset \PP^6$ from $O$ as in \cite{debarre2011nodal}.  

The quadrics containing $X_O$ form a net $P$. Let $D$ be the discriminant curve parameterizing singular quadrics. Then there is a line $L\subset D$ corresponding to quadrics containing the projection of $\Gr(2,V_5)$.
Therefore, $$D=C\cup L$$ for some degree $6$ curve $C\subset P$.

We see that the curve $\PP(V_{31})\cap Y_B^{\geq 2}$ parameterizing quadrics in $Y_B^{\geq 2}$ of corank $\geq 2$ whose vertices contain $O$, thus it equals $C$. Then by \cite[Proposition 2.20]{ferretti2009chow}, $C$  has $7$ nodes corresponding to the extra decomposable vectors. For a general $B$, $$Y_B^{=3}=\emptyset,$$ hence the quadrics in $C$ are all of rank $6$. There is a \'etale double cover $\tilde{C}\rightarrow C$, corresponding to the choice of a family of $3$-planes contained in a quadric in $C$, see \cite[section 4.2]{debarre2011nodal}. Let $p':\tilde{N} \rightarrow N$ be the normalization of $p:\tilde{C}\rightarrow C$.

There is a morphism $P^{\vee}\rightarrow C^{(6)}$, sending a line in $P$ to its intersection with $C$. Denote $S'$ the pull back of $P^{\vee}$ in $\tilde{C}^{(6)}$ and $S''$ the further pull back in $\tilde{N}^{(6)}$. 
By \cite[Proposition 5.8]{logachev2012fano}, $S'$ has two irreducible components and  $F(X_B)$ is birational to one of the components. It follows that $F(X_B)$ is birational to a component of $S''$. Then by \cite[Theorem 8.19]{welters1981abel}, $${\rm Alb} (F(X_B))\cong Pr(\tilde{N}/N),$$ which is of dimension $2$ since $g(N)=3$.
\end{proof}
\begin{lemma}
For a general point $y\in Y_A$, we can find  $B\in \Omega_A\cap \Sigma_8$, such that $y\in Y_B^{\geq 2}$.
\end{lemma}
\begin{proof}
The proof is similar to \cite[Proposition 5.1]{ferretti2012special}.
Let $$\Gamma=\Omega\cap \LL\GG(\bigwedge^3V_6)\times \Sigma_8\cap \LL\GG(\bigwedge^3V_6)\times \Omega_y,$$ where $\Omega_y=\{B~|~\mathrm{dim}(B\cap F_y)\geq 2\}$ and $F_y=y\wedge \bigwedge^2V_6$.  We have the two projections:

\begin{equation*}
\begin{tikzcd}
                       & \Gamma \arrow[ld, "\pi"'] \arrow[rd, "\rho"] &          \\
\{A\in\LL\GG(\bigwedge^3V_6)\mid y\in Y_A\} &                                              & \Sigma_8\cap\Omega_y 
\end{tikzcd}    
\end{equation*}
We choose a component of $\Sigma_8\cap\Omega_y$ whose general
point $B$ contains exactly eight independent decomposable forms
$\alpha_1,\ldots,\alpha_8$ and satisfies
$$
\dim(B\cap F_y)=2,
~~
B=\langle\alpha_1,\ldots,\alpha_8\rangle\oplus(B\cap F_y).
$$
The chosen component of $\Sigma_8\cap\Omega_y$ is smooth of
dimension $55-8-3=44$ at a general point.
The fibers of $\rho$ are $\Omega_B$, of dimension $10$. Indeed. 
one can chooses a hyperplane $U\subset B$ and a Lagrangian
subspace $A\neq B$ containing $U$.
Hence the corresponding component of $\Gamma$ has dimension
$54$, equal to that of the irreducible divisor
$\{A\mid y\in Y_A\}$.
It suffices to show that $d\pi$ is injective at a general point
of this component.

Let $(A,B)\in \Gamma$ be general such that the $8$ decomposable forms $\alpha_1,\ldots \alpha_8$ in $B$ are linearly independent and that $\alpha_1,\ldots \alpha_8$ and $F_y\cap B$ are not contained in $A$. Let $U=A\cap B$. According to \cite[Lemma 5.4]{ferretti2012special}, we have the description of tangent spaces:
\begin{equation*}
\begin{split}
    T_{(A,B)}\Omega &=\{(q_A,q_B)\in \mathrm{Sym}^2(A^{\vee})\times \mathrm{Sym}^2(B^{\vee})~|~q_A|_U=q_B|_U\}\\
    T_B\Sigma_8 &=\{q_B\in \mathrm{Sym}^2(B^{\vee})~|~q_B(\alpha_1)=\ldots =q_B(\alpha_8)=0\}\\
    T_B\Omega_y&=\{q_B\in\operatorname{Sym}^2(B^\vee):q_B|_{B\cap F_y}=0\}.
\end{split}
\end{equation*}

We assume on the contrary that there is a nonzero $t=(q_A,q_B)\in \mathrm{Ker}~d\pi$. It implies that $q_A=0$ and thus Ker $q_B=U\cup U'$ for some hyperplane $U'\subset B$. According to the assumptions and the description of tangent spaces, the $8$ decomposable forms and $F_y\cap B$  are contained in Ker $q_B$ hence contained in $U'$. Then for a general $B$,  dim $U'=10$, which is a contradiction. 
\end{proof}

Now we can deduce the main result of this subsection:
\begin{proposition}\label{conic}
For any conic $c\subset X$, we have $c\in \rS_2(X)$.
\end{proposition}
\begin{proof}
We may assume that $c$ is a general conic.
Combining the above  Lemmas, there exist a GM $3$-fold $X_B$  with $8$ nodes containing $c$ and a period partner $X'$ of $X$ such that $X_B\subset X'$. Let $S'$ be the image of $$F(X_B)\subset F(X)\xrightarrow{\alpha} \tY^\vee$$ and $\tilde{S}$ be the resolution of $S'$. Then a conic lying in $R'$ is in $\rS_1(X)$, where $R'$ is the normalization of $S'\cap D$.

The argument of \cite[Section 2.4]{shen2020categories} implies that there is a surjection $$R'^{(2)}\twoheadrightarrow Alb(\tilde{S}).$$  
For a resolution $\tilde{X}_B\rightarrow X_B$, we have $\CH_1(\tilde{X}_B)_{\rm hom}\cong J(\tilde{X}_B)$ by \cite[Theorem 0.3]{voisin2013abel}. Hence, $\CH_0(\tilde{S})_{\rm hom}\rightarrow \CH_1(\tilde{X}_B)_{\rm hom}$ factors through $Alb(\tilde{S})$.  By the surjectivity of $R'^{(2)}\rightarrow Alb(\tilde{S})$, 
there exists two conics $c_1$ and $c_2\in \rS_1(X')$ such that 
$c=c_1+c_2-2\theta\in \CH_1(\tilde{X}_B)$. Then the result follows from Corollary \ref{partner}.
\end{proof}

\section{Sheaves supported on conics and lines}\label{S5}
In this section, we introduce a link between the filtration on $\CH_1(X)$ and the Kuznetsov component $\cA_X$ of the derived category of $X$ and prove Theorem \ref{mainthm}. We assume that $X$ is very general.

Let $i^*: D^b(X)\rightarrow \cA_X$ be the left adjoint of the inclusion $i_*:\cA_X\rightarrow D^b(X)$ and $\pr: K_0(X)\rightarrow K_0(\cA_X)$ be the projection of Grothendieck groups.  Let $\overline{\pr}$ be the further projection to $K_{\rm num}(X)$. For a very general $X$, the numerical Grothendieck group $K_{\rm num}(X)\cong \ZZ^{\oplus 2}$ and under the basis $\lambda_1$, $\lambda_2$ the Euler form is given by (cf. \cite[Lemma 2.4]{pertusi2019double}): 
$$\begin{pmatrix}
-2 & 0 \\ 0 & -2
\end{pmatrix},$$
where $\lambda_1=\overline{\pr} [\cO_c(1)]$ and $\overline{\pr} [\cO_p]=-\lambda_1-2\lambda_2$ for a conic $c\subset X$ and a point $p\in X$.  

\begin{lemma}\label{K0}
$K_0(\cA_X)\otimes \QQ$  is generated by $\pr [\cO_c(1)]$ and $\pr [\cO_p]$.
\end{lemma}
\begin{proof}
Since $\CH_0(X)\cong\ZZ$, the class $[\cO_p]$ in $K_0(X)$ is independent of the choice of $p\in X$ and the cycle class map $cl: \CH^i(X)\rightarrow \rH^i(X,\ZZ)$ is injective for $i\neq 3$, by \cite[Theorem 1.1]{bloch1983remarks}. 

We have the following commutative diagram:
$$\begin{tikzcd}
K_0(\cA_X)\otimes \QQ \arrow[d, "\overline{\pr}"] \arrow[rr, "\ch"] &  & \CH^*(X)\otimes \QQ \arrow[d, "cl"] \\
K_{\rm num}(\cA_X)\otimes \QQ \arrow[rr, "\overline{\ch}"]                         &  & {\rH^*(X,\QQ)}    ,       
\end{tikzcd}$$
It's enough to show that $cl^{-1}(\overline{\ch}(K_{\rm num}(\cA_X)))$ is generated by $\ch(\cO_c(1))$ and $\ch(\cO_p)$ and this follows from the injectivity of the cycle class map for $i\neq 3$ and the fact that  $\CH_1(X)\otimes \QQ$ is generated by conics.
 \end{proof}
For a cubic fourfold, the link between the Kuznetsov component and 1-cycles is by taking the Chern class $c_3$. However, the situation is slightly different for GM fourfolds:

If $X$ and $X'$ are period partners, their associated dual double EPW sextics are the same by Theorem \ref{moduli}.
Therefore, their Fano varieties of conics both admit a birational $\PP^1$-fibration to $\tY^\vee$. Then by Lemma \ref{K0} there is a group isomorphism $$K_0(\cA_X)_\QQ\cong K_0(\cA_X')_\QQ,$$ which maps $\pr [\cO_p]$ to $\pr [\cO_{p'}]$ and $\pr [\cO_c(1)]$ to $\pr [\cO_{c'}(1)]$. Here $c'$ is a conic mapping to the same point with $c$ in $\tY^\vee$ via the fibration. 

However, the classes $c_3(\pr [\cO_p])$ and $c_3(\pr [\cO_{p'}])$ do not coincide in $\CH_0(\tY^\vee)$ via the isomorphism \eqref{chowiso} even modulo the class $o$. We want the link between the Kuznetsov components and 1-cycles intrinsic for the EPW sextic, so
we modify the $c_3$ to be compatible with \eqref{chowiso} by dropping the term of $\pr [\cO_p]$: 

If $[\cE]=\sum_i a_i\pr[\cO_{c_i}(1)]+b\pr[\cO_p]\in K_0(\cA_X)_\QQ$, we set $$p(\cE)=\sum a_ic_i\in \CH_1(X)\otimes \QQ.$$ It is independent of the choice of representations, since $b$ relies only on the numerical class of $[\cE]$. 
Let $\cF$ be a sheaf supported on a nonsingular connected rational curve $C\subset X$ of degree $e>0$. Then $[\cF]$ is numerically equivalent to $$re[\cO_l(1)]+m[\cO_p]$$ in $K_{\rm num}(X)$ for some integers $r>0$ and $m$, where $l$ is a line on $X$. Let 
$$d(i^*\cF):=\frac{1}{2}\dim \Ext^1_{\cA_X}(i^*\cF,i^*\cF)$$ and we have:

\begin{lemma}
If $\cF$ is supported on a line or a conic, then $d(i^*\cF)\geq 2$.
\end{lemma}
\begin{proof}
By straightforward computation, we have $\overline{\pr} [\cO_p]=-\lambda_1-2\lambda_2$ and $\overline{\pr} [\cO_l(1)]=-\lambda_2$. Hence $\overline{\pr} [\cE]=-m\lambda_1-(2m+re)\lambda_2$.
Then 
$$d(i^*\cF)\geq \frac{1}{2}\chi (i^*\cF,i^*\cF)+1=5m^2+4mre+r^2e^2+1.$$
For $e=1$ or $2$, we can easily deduce that $d(i^*\cF)\geq 2$. In fact, there is a lower bound by a quadric polynomial in $e$, but we only deal with $e=1$ or $2$.
\end{proof}

Now we can prove the main theorem:
\begin{theorem}\label{maint}
If $\cF$ is supported on a line or a smooth conic and $\cE=i^*\cF$, then
\[
p(\cE) \in \rS_{d(\cE)}(X).
\]
\end{theorem}
\begin{proof}
By the definition of $p$, it is enough to show that every line or conic belongs to $\rS_2(X)$. For conics, it is proved in Proposition \ref{conic}. For lines, the strategy is to degenerate the residual curve of a line on a certain surface to a sum of conics and lines.

The locus in $F(X)$ and its image in $\tY^\vee$ parameterizing double lines is a surface  by \cite[Lemma 3.8]{iliev2011fano}. Since $X$ is very general,  this surface meets $Z$  by Proposition \ref{o}. Hence, there exists a line $l_0\subset X$ such that $2l_0=2 \theta$. Take $\PP(V_1'\bigwedge V_3')=l'=\iota_1(l_0)$,  we obtain
$$2l'=2\sigma-2\theta,$$
where $\iota_1$ is the rational involution of the Hilbert scheme of lines $F_1(X)$.

For $l=\PP(V_1\bigwedge V_3)\subset X$ a general line, let $V_2=V_1\oplus V_1'$ and $V_4=V_1'\oplus V_3$ and $S_{V_2}$ be the surface parameterizing points $x\in X$ with  $V_x\cap V_2\neq \emptyset$, where $V_x$ is the two dimensional vector space corresponding to $x$. We see that $S_{V_2}$ is a degree $6$ surface in $\PP(V_2\bigwedge V_5)\cap H=\PP^5$.
By construction, $l$ and $l'$ are contained in $S_{V_2}$ and $S_{V_2}\cap S_{V_4}$ is a hyperplane section of $S_{V_4}\subset \PP(\bigwedge^2V_4)\cap H$, which is the union of $l$ and a degree $3$ curve $C_3$. Hence
$$l+C_3=4\theta .$$ 
Since $V_3'\cap V_4$ is 2-dimensional, $l'$ meets $S_{V_4}$ at the point corresponding to $V_3'\cap V_4$. For a general choice of $l$, it does not meet $l'$. It follows that $l'$ and $C_3$ meet at a point, and therefore they span a $\PP^4\subset \PP(V_2\bigwedge V_5)\cap H$, which cuts $S_{V_2}$ into a degree $6$ curve $l'\cup C_3\cup c$ for some conic $c$. 

We claim that 
$$l'+C_3+c=4\theta+\sigma.$$
In fact, we can choose the hyperplane section $\PP^4\subset \PP(V_2\bigwedge V_5)\cap H$ to be the span of  $\PP(V_2\bigwedge V_4)\cap H$ and $\PP(V_1\bigwedge V_1^{\perp})$ for some $V_1\subset V_2$ with $V_1^{\perp}\neq V_4$. Then the class of a section is $l+C_3+c'$, which equals $4\theta+\sigma$. Here $c'$ is the $\sigma$-type conic associated with $V_1$.

Finally, we obtain that $l=l'+c-\sigma=c-l_0$, which is in $S_2(X)$ by Proposition \ref{conic}.
\end{proof}


\makeatletter
\enddoc@text
\global\let\enddoc@text\@empty
\makeatother
\clearpage
\providecommand{\OO}{\mathcal O}
\providecommand{\UU}{\mathcal U}
\providecommand{\YY}{\widetilde Y^{\vee}}
\theoremstyle{plain}
\newtheorem{correctionlemma}{Lemma}[section]
\appendix
\section{A fix to the gap in Lemma 4.12}
\label{app:mazzanti-correction}
\begin{center}\textsc{Carl Mazzanti}\end{center}
\markboth{CARL MAZZANTI}{A FIX TO THE GAP IN LEMMA 4.12}
\thispagestyle{plain}

We use the notation of the main text. Throughout, $S_{V_4}$ is
 general, $c=c_i\in L_i$, $c_0\in L_i$ is a second conic disjoint
from $c$, and $\iota(c)$ denotes a conic of the residual pencil $L_{i+5}$.

It is claimed in Lemma \ref{KeyLemma} that
\[
 D_{c_i}^2=\sum_{j\ne i,i+5}L_j+4L_{i+5}+kL_\sigma\in\CH_1(F).
\]
In the proof, a conic $c'\not\subset S_{V_4}$ meeting $c$ in $x$ and
$c_0$ in $y$ with $\dim(V_x\cap V_y)=1$ is shown to be a
$\sigma$-conic, using that the line through $x$ and $y$ lies in
$\langle c'\rangle\cap\Gr(2,V_5)$ together with
Lemma \ref{plane}. However, this line may be a component of
$c'$, in which case Lemma \ref{plane} gives no contradiction.
More precisely, let $\ell\subset S_{V_4}$ be a line with $\ell\cdot c=1$,
and denote by
\[
 R_\ell:=\overline{\{\ell\cup\ell'\mid\ell'\subset X
       \text{ a line meeting }\ell,\ \ell'\ne\ell\}}\subset F
\]
the curve of reducible conics containing $\ell$.
Then $R_\ell\subset D_c\cap D_{c_0}$, its general member is a
$\tau$-conic, and $R_\ell$ is not contracted by $\alpha$, since the
$4$-dimensional span of $\ell\cup\ell'$ varies with $\ell'$.
These components are missing from the statement.

Writing $S_{V_4}=\operatorname{Bl}_5\PP^2$ and $c=H-E_1$, the lines
meeting $c$ are $E_1$, $H-E_j-E_k$ for $2\le j<k\le5$, and
$2H-\sum_jE_j$, while the lines meeting $\iota(c)$ are the remaining
eight. Furthermore, every line meets exactly one of $c_i,c_{i+5}$
and exactly five of $c_1,\ldots,c_{10}$. This has two consequences.
\begin{enumerate}
\item The identity $G\cdot(D_c^2-D_{\iota(c)}^2)=4m(\iota(t)-t)$
in the proof of Proposition \ref{decomposition} acquires the term
\[
 G\cdot\left(\sum_{\ell\cdot c=1}R_\ell
                 -\sum_{\ell\cdot\iota(c)=1}R_\ell\right),
\]
so the proof of injectivity is incomplete as written.
\item Formula (4.8) acquires the term
$\sum_{\ell\subset S_{V_4}}R_\ell$ (up to the multiplicity of $R_\ell$),
which enters the proof of $a\ne0$ in Lemma \ref{nonzero}
and hence Theorem \ref{constant}.
\end{enumerate}
The second part of Lemma \ref{nonzero} only uses
Proposition \ref{formula} and is unaffected. We give the following correction to Lemma \ref{KeyLemma} and both two consequences above.

\subsection*{The correction}
\begin{correctionlemma}\label{corr:lem:intersection}
For a general conic $c_i$ on a general $S_{V_4}$, we have
\begin{equation}\label{corr:eq:intersection}
 D_{c_i}^2=\sum_{j\ne i,i+5}L_j+4L_{i+5}+kL_\sigma
                 +\sum_{\ell\cdot c_i=1}R_\ell\in\CH_1(F).
\end{equation}
\end{correctionlemma}
\begin{proof}
It remains to show that $D_c$ and $D_{c_0}$ meet transversally at a
general point $c'=\ell\cup\ell'$ of $R_\ell$. Let $p:=\ell\cap\ell'$,
$x:=\ell\cap c$ and $y:=\ell\cap c_0$, which are three distinct points.
As in the proof of Lemma \ref{KeyLemma}, we have
\[
 T_{c'}D_c=\{s\in H^0(N_{c'/X})\mid s(x)\in N_{\ell/S_{V_4},x}\},
\]
and similarly for $c_0$ and $y$. Let
$\mathcal Q:=N_{c'/X}|_\ell/N_{\ell/S_{V_4}}$.
For a general $\ell$ as above, a general line $\ell'$ meeting $\ell$
has $N_{\ell'/X}\simeq\OO^{\oplus3}$.
Hence $N_{c'/X}|_{\ell'}(-p)\simeq\OO(-1)^{\oplus2}\oplus\OO$, and,
since $N_{\ell/S_{V_4}}\simeq\OO(-1)$, the map
$H^0(N_{c'/X})\to H^0(\mathcal Q)$ is surjective with one-dimensional
kernel $T_{c'}R_\ell$. The sheaf $\mathcal Q$ is the elementary
modification at $p$ of
$N_{S_{V_4}/X}|_\ell=\UU^\vee|_\ell\simeq\OO\oplus\OO(1)$
in the direction of the image of $T_p\ell'$.
This image is the linear form on $U_p$, where $U_p$ is the
two-dimensional subspace associated to $p\in X$, vanishing on the pivot
of $\ell'$, while the subbundle $\OO(1)$ consists of the forms vanishing
on the pivot of $\ell$. Since the two pivots differ,
$\mathcal Q\simeq\OO(1)^{\oplus2}$ and
$H^0(\mathcal Q)\to\mathcal Q_x\oplus\mathcal Q_y$ is an isomorphism.
Thus $T_{c'}D_c\cap T_{c'}D_{c_0}=T_{c'}R_\ell$, meaning that the
intersection is transversal along $R_\ell$.
\end{proof}

\subsection*{Lemma 4.13, first part}
Let $\Gamma\subset F_1(X)\times F$ be the closure of the incidence
$\{(\ell,\ell\cup\ell'):\ell'\ne\ell,\ \ell\cap\ell'\ne\varnothing,
\ \ell\cup\ell'\text{ is of }\tau\text{-type}\}$,
with its reduced four-dimensional cycle. By the lemma, the additional term in (4.8) is
\[
 \sum_{\ell\subset S_{V_4}}R_\ell
       =\Gamma_*\left(\sum_{\ell\subset S_{V_4}}[\ell]\right).
\]
The zero-cycle $\sum_{\ell\subset S_{V_4}}[\ell]\in\CH_0(F_1(X))$
is parametrised by $V_4\in\PP(V_5^\vee)$, a rational variety, so it
does not depend on $V_4$. Hence the additional term is constant in
$\CH_1(F)$, and the proof of $a\ne0$, and thus of Theorem \ref{constant}, goes through.

\subsection*{Proposition 5.5}
Pulling back the identity $I^2=aW+I\cdot A+B+C$ of Proposition \ref{formula} along
$\{c\}\times F\hookrightarrow F\times F$ for a general $\tau$-conic $c$
gives
\[
 D_c^2=aL_{\iota(c)}+A_2\cdot D_c+B_2,
\]
where $A_2\in\CH^2(F)$ and $B_2\in\CH_1(F)$ do not depend on $c$.
Here we use that $q$ is flat Lemma \ref{L2.3}, that
$W_*[c]=L_{\iota(c)}$, and that $C$ is supported on $\Sigma_2$.
Since $D_c+D_{\iota(c)}$ is constant by Lemma \ref{2.1},
so is $D_1:=D_c+D_{\iota(c)}-A_2$, and
\[
 G\cdot D_1\cdot(D_c-D_{\iota(c)})=am(\iota(t)-t)
\]
under the identification $\alpha_*:\CH_0(F)\simeq\CH_0(\YY)$.
The proof of injectivity in Proposition \ref{decomposition} then
goes through with $D_0$ replaced by $D_1$ and $4$ replaced by $a$,
which is non-zero by the above, completing the proof.


\end{document}